\documentclass{article}

\usepackage{stmaryrd,amssymb,amsmath,latexsym,amsthm,url,xfrac}
\usepackage[english]{babel}
\usepackage{eurosym}

\usepackage[usenames]{color}
\usepackage{amsthm}

\newtheorem{claim}{Claim}[section]
\newtheorem{principle}{Principle}[section]
\newtheorem{perspective}{Perspective}[section]
\newtheorem{proposition}{Proposition}[section]
\newtheorem{definition}{Definition}[section]

\newcommand{\isaq}{=_{ \small \mathsf{AQ}}}

\newcommand{\Nat}{{\mathbb N}}
\newcommand{\Int}{{\mathbb Z}}

\providecommand{\dotdiv}{
  \mathbin{
    \vphantom{+}
    \text{
      \mathsurround=0pt 
      \ooalign{
        \noalign{\kern-.50ex}
        \hidewidth$\smash{\cdot}$\hidewidth\cr 
        \noalign{\kern.50ex}
        $-$\cr 
      }%
    }%
  }%
}

\title{Sumterms, Summands, Sumtuples, and Sums and the Meta-arithmetic of Summation\\
}

\author{%
Jan A.\ Bergstra\\
{  Informatics Institute, University of Amsterdam}%
\thanks{email: \texttt{j.a.bergstra@uva.nl, janaldertb@gmail.com}.}
\date{\small{September 18, 2020} }
}

\begin{document}

\maketitle

\begin{abstract} Sumterms are introduced as syntactic entities, and sumtuples are introduced as semantic entities. 
Equipped with these concepts a new description is obtained of the notion of a sum as (the name for) a role 
which can be played by a number. Sumterm splitting 
operators are introduced and it is argued that without further precautions the presence of these operators 
gives rise to instance of the so-called sum splitting paradox. 
A survey of solutions to the sum splitting paradox is given.
\end{abstract}

\newpage
\tableofcontents

\section{Introduction}
I will use the phrase meta-arithmetic to refer to reflections about arithmetic, 
without the connotation, known from
meta-mathematics that such reflections are necessarily supported by or even consisting of mathematical 
work, or mainly consist of formal logic.

With elementary arithmetic I will understand the activity of developing, confirming and disconfirming, proving and disproving, 
open and closed identities over an arithmetical signature.  I will restrict attention to the case of characteristic zero, and in fact
I will understand arithmetic as school arithmetic rather than as its academic (somehow advanced) counterpart.

The central question of this paper is: ``what is a sum?'' Unlike the related question ``what is a fraction'' 
the educational literature leaves this question remarkably untouched. 
I propose to understand sum a a role name, rather than as a noun. in arithmetic, the role of a sum is played by a number.

\subsection{Sum, left summand, and right summand as role names}
\label{Roles}
To the best of my knowledge there is no convincing definition of sums as a collection of entities. 
In this respect sums differ for instance from squares which may, but need not, be understood as a collection of numbers.  Sum is a role name, as in ``$a$ is the sum of $b$ and $c$'', rather than a noun
referring to an entity of a certain type. Moreover, in this case $b$ has the role \emph{left summand}, and $c$ has the role \emph{right summand}. Further $a$ has as well the role  \emph{function value}, and both $a$ and $b$ have the respective 
summand roles and in addition to those role(s) the still weaker  role of (being an) \emph{argument}. 

Objects, concrete as well as abstract, may have different roles at the 
same time, and the same object may have different roles in different contexts. Unlike physical objects one and the same 
number can appear in different contexts at the same time. For instance in $3 = 2 + 1$ and $5 = 2 + 3$, 
two signs which coexist in time, the sub-sign $3$ has different roles.

Given $a <b$, $a$ has the role of the smaller entity in an ordered pair. It is pointless to say that $a$ is small (or smaller).
Indeed $a$ is element of an ordered domain, which, however, is not the same as its  role in $a <b$. Sum, summand, left summand, 
and right summand are roles which are plausible for numbers. Who thinks of sum and summand as roles, thinks of numbers as entities capable of fulfilling such roles. The view that sum is a role instead of a predicate, or a collection denoted by a predicate, therefore comes with some form of realism on numbers, 
because as entities these may fulfil one or more roles.

\subsection{Sum: an abstract parametric proper noun}
I will assume that say $75$ is a proper noun which stands for a sufficiently specific entity. Proper nouns are also called proper names. Alternatively, and less plausible, one may contemplate the notion that $75$ is a common noun (also called generic noun) so that it stands for all entities which somehow qualify as the number $75$. Suppose it is 
agreed that $x$ is a variable which ranges over natural numbers (whatever these are, even if they don't exist after all), one may holds that, within a context aware of said agreement, $x$ is a common noun but not a proper noun. 
Now consider $75+x$. This sign may be considered to denote a parametric proper noun: 
a noun which comes about after obtaining a denotation for a parameter $x$ from an environment. 
In fact $x$ itself may also be considered a parametric
proper noun. While ``natural number'' is a common noun ranging over the natural numbers, 
$x$ names a not yet known, but once known also specific number. 
In logic the context providing an actual number to which $x$ refers is called a valuation.

Sum is not a proper noun (which must refer to a specific entity, also called a generic noun), as it does not refer to any specific entity. But sum is not a common noun (ranging over some  class of entities, alternatively called a generic noun) either. I will qualify sum as a parametric proper noun: in the context of two summands sum is a proper noun with the summands as parameters. The summands act as parameters to be taken from a context 
which allow ``sum'' to refer to a specific value.

Roles which are played by unique entities once key paramaters from a context have been determined may be considered parametric proper nouns.

Nouns may also be qualified in terms of concreteness, where concrete nouns range over collections of physical entities, 
and where abstract nouns range over non-physical entities. So ``natural number'' may be considered an 
abstract common noun, and sum may be considered an abstract parametric proper noun. The sign $25$ is a
physical entity, which may be uniquely determined in a given context, say on a piece of paper. 
Then ``look at the $25$ in the middle of the page'' contains $25$ in the quality of a singular 
concrete proper noun. In ``the $25$'s in this paper should be rendered boldface'' $25$ is embedded in a a plural noun.
As it turns out 
the typing of grammatical units as signs (i.e. entities), nouns, proper nouns, common nouns, 
abstract nouns, concrete nouns etc. is non-trivial and may give rise to polymorphism.

A distinction between collective nouns and singular nouns may be made, besides and independently of, 
 a distinction between plural nouns and singular nouns. Then ``the arguments of function $f(-,-,-)$'' and ``the values of function $g(-,-,-)$'' are collective nouns. 
Plural nouns are nouns which occur in plural form, for instance the 
occurrence of  \emph{times} in ``$3$ is three \emph{times} as large as $1$''. Plural noun is a linguistic notion without
manifest logical content.

I summarize the above considerations regarding sums with the following claim.
\begin{claim} (Scepticism on sums) 
There is no convincing and generally accepted definition of the concept of a sum in elementary arithmetic. 
The notion of a sum is too ambiguous for it to be used as a common noun or as a noun.
\end{claim}

\subsection{Sumtuples}
A sumtuple is an ordered triple $(a,b;c)$ such that $c$ is the sum of $a$ and $b$. So a sumtuple, in this case
also called a sumtriple, contains a sum, and in somewhat less accurate language it has a sum. 
The sum of a sumtuple is simply its third component.
Once the terminology of sumtuples has been coined one may imagine a distinction between valid sumtuples and invalid sumtuples. For a valid sumtuple $(a,b;c)$ it is the case that $a+ b = c$, and for an invalid sumtuple that is not the case. In some cases  a more detailed notation for sumtuples may be useful, e.g. $+(a,b;c)$ or $\mathsf{plus}(a,b;c)$.

The pair $(a,b)$ of both summands of a sumtuple is important as well and I will refer to such a pair as 
a sumterm or as a plusterm. Below we will only use sumterm. 
Different notations for sumterms are useful, e.g. $+(a,b)$ or simply $a+b$. It is sumterms rather than sums 
which are equipped with two summands. Sumterms have a sum as well, not in the sense of containing a component which has the role of a sum, but in the sense of  allowing a unique completion with a sum (a number with role sum) to a valid sumtuple. As it tuns out ``having a sum'' is ambiguous (or rather polymorphic), it means something different for 
sumterms and for sumtuples.

Conceiving a function as a graph its tuples contain result values, whereas upon viewing a function as a 
prescription there is no notion of containment of a result. 
Having a value is a different matter for different conceptions of function, and the  mentioned
divergence of interpretation for having a sum correlates with different ways of understanding addition.

\subsection{Looking for more detail}
My objective in this paper is to sharpen the picture of sums which was sketched in the preceding Paragraphs. 
Various questions immediately arise:
(i) what is a number, in other words which entities play the roles of sum and of the respective summand, 
(ii) what is the relation between syntax and semantics, 
(iii) is any of the two, numbers or expressions for numbers, more fundamental than the other, 
(iv) why not do without expressions (sumterms), 
(v) why not do without 
sumtuples, 
(vi) what is the official academic perspective on these matters, if any, 
(vii) what happens if one insists that a sum is a number with two components, a left summand and a right summand, 
(viii) what is the status of sums of more than two numbers (are these roles just as well, or do nested expressions, or abbreviations thereof,  enter the picture at this stage)?
\section{Sumterms a kind of arithmetical quantities}
 Sumterm is a noun ranging over expressions with addition as the leading function symbol. 
 Using the notion of a sumterm requires acceptance of a distinction between syntax and semantics. 
 I   will refer to an expression or term as an arithmetical quantity (abbreviated to AQ below). 
 \begin{definition} 
 \label{SumtermDef} A sumterm is an AQ with addition (usually written $\_+\_$) as its leading function symbol.
 \end{definition}
 With this definition comes some level of abstraction: $2+3,  \mathsf{plus}(2,3), +(2,3)$ each present the same sumterm, 
 in the sense of Definition~\ref{SumtermDef}. Differences in typography and formatting differences 
 are considered inessential. Sumterms are not signs but may be abstractions (interpretations, denotations, meanings) 
  of signs.
 
 \begin{definition}

 For a sumterm $p + q$ the AQ $p$ is its left (or first) summand and the 
 AQ $q$ is its right (or second) summand.
 \end{definition}
Sumterms have a sum, which is another word for value in the special case of sumterms. 
As was already stated above, unlike sumtuples, however, sumterms do not contain a sum.

 One may think of natural numbers as presenting straightforward context in 
 which to contrast syntax and semantics.  Remarkably, however, 
 maintaining a convincing distinction between syntax for naturals and semantics of naturals  is challenging. 
 
 I assume that
 it is plausible to see some form of syntax at work for an agent who makes sense of the sign $2+ (5 + ((-3) +0))$. It
 is less obvious, however, what may count as semantics in relation to such kind of signs.  
  The notion of an AQ comes with a syntactic bias just as the notion of a number comes with a semantic bias. 
  Both for numbers and for AQs,   however,  providing  an unambiguous definitions seems to be impossible. 
    
  AQs are abstract entities, not tangible signs. 
  AQs may serve as interpretations of signs, 
  and more specifically of arithmetical signs, 
  amenable to being conceived of as abstract entities. 
  Conversely signs may represent AQs or parts thereof 
  in a human readable manner.
 A sign may denote an AQ, but the AQ does not inherit attributes like, size, place, time, color, authorship, etc. form that sign.
  I will assume that meta-arithmetic takes  for granted the existence of AQs which can be imagined as abstract entities, 
  in advance of any detailed understanding of arithmetic. 
  
  Semantics enters the picture via the assumption that for AQs $p$ and $q$, the identity (equation) 
  $p=q$ expresses the notion that $p$ and $q$ have  the same meaning (or equivalently: denote the same entity). 
  Semantics   concerns questions about the nature of meanings for various classes of AQs?
 Semantics has two aspects: ontology: what are semantic entities, and denotation: 
 which semantic entity is assigned to a given AQ. For AQ $p$ the latter 
 entity is denoted with $\llbracket p \rrbracket$, which is itself a notation from meta-arithmetic, unlikely to show 
 up in arithmetic proper.
 
 \begin{definition} A poly-infix sumterm is an AQ of the form $t_1+ \ldots + t_n$ for natural $n >2$. 
 \end{definition}
 Poly-infix sums of naturals require naturals to be known for indices. It is possible to introduce poly-infix 
 operations for each arity one at a time, so that in some stage only a limited number of 
 such operations have been defined and no appeal is made to an infinite totality of natural numbers.
 
  \begin{definition}
 Let $k,n$ be decimal natural numbers (see~\ref{DecNats} below)  with $k \leq n$. Then
 
 (i)  length of the poly-infix sumterm $t_1+ \ldots + t_n$ is $n$, and 
 
 (ii)  the $k$-th summand (also denoted $\mathsf{smnd}_k(p)$)  of the poly-infix sumterm $p \equiv t_1+ \ldots + t_n$,  is $t_k$.
 \end{definition}
 For $n=2$, a poly-infix sumterm is just a sumterm, with $\mathsf{smnd}_1(p)$ as its left summand, and 
 $\mathsf{smnd}_2(p)$ as its right summand. 
 The sums of poly-infix sumterms of different arities are related as follows: \\
 $t_1+ \ldots t_{n-1}+ t_n + t_{n+1}= t_1+ \ldots +  t_{n-1}+(t_n + t_{n+1})$.

Adopting the presence of poly-infix sumterms can be contrasted with adopting the view that say $t_1+ \ldots + t_4$ abbreviates $((t_1+ t_2) + t_3) + t_n$. By using poly-infix sumterms the use of inductively defined syntax can be 
delayed or even avoided. 

The introduction of poly-infix sumterms suggests the introduction of abbreviations as follows: e.g. 
$t_1- t_2 -t_3 + t_4$ abbreviates the sumterm $t_1+(- t_2) +  (-t_3) + t_4$. 

Poly-infix sumterms arise in the teaching of arithmetic if for instance AQs like $3 + 7 + 9$ appear without either a
 description of a convention on how to view the expression as an abbreviation of another expression 
 which involves additional brackets, or alternatively an explanation on when brackets may be left out, detailing the role of assiciativity for that matter.

 \begin{definition}
 \label{sumDef}
 The sum of AQs $p$ and $q$ is the value of the sumterm $(p)+(q)$.
 \end{definition}
 For instance: the sum of $p\equiv 1+2$ and $q \equiv 3+4$ is the value of $(1+2) + (3+4)$.
 In view of the fact that $(p) + (0) = (p) + 0  = p$, all numbers are such values, 
 and for that reason Definition~\ref{sumDef} does not lead to a non-trivial notion of sum (as a proper noun). 
 This situation may be contrasted 
 with squares which can be usefully defined, among integers, as the members of the range of the squaring operation.

 Given an equation $p = Q$ with $Q$ a sumterm or a poly-infix sumterm it may be said that $p$ is written as a sum. 
 Thus ``written as a sum'' applies to AQs and abbreviates ``written as a sumterm or written as a poly-infix sumterm''.

 \subsection{Semantics of arithmetical AQs}
 I will discuss various ways in which one may imagine natural numbers as abstract entities. 
 Remarkably the seemingly elementary question ``what is a natural number'' 
 has not received a definitive answer in mathematics, in logic, or in philosophy. 
 The conventional understanding of natural numbers as finite ordinals in ZF set theory, 
 thus following the classic proposals of J. von Neumann,  
 seems not to provide an intuitive understanding of natural numbers 
 which is both convincing and  is usable at an elementary level at the same time.

 In~\ref{SurveySem} below I will survey different options for the ontology of natural numbers and 
 in addition I will motivate the choice just presented:
 in terms of ontology, natural numbers are AQs and as a semantics producing operation finding the meaning of an AQ 
 with type natural number is a projection.
 
Rather than leaving the ontology of natural numbers open, 
 I will choose as my preferred semantics of natural numbers a collection of 
 arithmetical quantities, the so-called decimal naturals (see Paragraph~\ref{DecNats} below). The semantic function $\lambda t. \llbracket t \rrbracket$ is a 
 projection (for each AQ $t$ it is the case that $\llbracket  \llbracket t \rrbracket \rrbracket =  \llbracket t \rrbracket $).
  For example: $\llbracket \llbracket 17 + (-1) \rrbracket \rrbracket =  \llbracket 16 \rrbracket = 16$.

\subsection{Arithmetical quantities, substitution and let-constructs}
One may speak of an arithmetical quantity $t$ without having a precise definition of the relevant syntax at hand. 
So $1+2$ can be labeled an AQ even if (that is at a stage where) there is no
answer to the questions like: (i) is $(1) +2$ an AQ, (ii) is $((1)) +2$ an AQ , (iii) is $(1 +2)$ an AQ, and (iv) is $+(2,3)+4$ an AQ? 
Below it will be assumed that these questions are given an affirmative answer each, but that need not be the case.

On AQ there is an equality $\isaq$, abstracting from redundant bracketing and spacing, 
including new lines and new pages. For instance $0 \isaq (0) \isaq ((0))$ and
$1+ 2 \isaq 1 + (2) \isaq (1+2) \isaq ((1) + 2)$ while $1+2 \not \isaq 2+1$ and$1+2 +5 \not \isaq (1+2) + 5$. 
Substitution on AQs 
binds stronger than addition and it will introduce additional brackets as follows: $[t/x]P[X] \isaq P[(t)]$. For instance $[1+2/X] (3+X) \isaq (3 + (1+2)) \isaq 3 + (1+2)$,
while $[1+2/X] 3+X \isaq (3 + X) \isaq 3 + X$. The let construct, however, 
works differently, 
by not introducing any additional brackets: 
$\mathsf{\,let\,} x \isaq 1+2 \mathsf{\,in\,} (3 + x) \isaq (3 + 1 + 2) \isaq 3+ 1 + 2$. 
Substitution and the let construct are connected as follows: $[t/x] A \isaq \mathsf{\,let\,} x \isaq (t) \mathsf{\,in\,} A$.

In general, for closed AQs $X$ and sumterms $Y$, $X \isaq Y \implies X=Y$ but not the other way around.

\subsection{Balance principle for arithmetical quantities}
Working with AQs amounts to the adoption of some form of syntax. However, adopting syntax without 
meaning is problematic for arithmetic. There must be a balance between syntactic considerations and semantic considerations. If $X$ is considered an expression, or an AQ, the question arises what kind of entity is denoted by $X$.
I will assume  the following principle.
\begin{principle} (Balance principle for syntax and semantics.) 
Adopting AQs or any other form of syntactic entities, 
brings with it a complementary need to adopt semantic considerations.
\end{principle}
Given the balance principle one knows that any explanation of arithmetic 
making use of AQs will at the same time make use of some 
notions of natural number and integer number, 
which refer to the respective collections of semantic entities.

\begin{claim} 
\label{SPCss}
(Skewed plurality claim for syntax and semantics.) 
Once a contrast between syntax and semantics is adopted for 
elementary arithmetic it is plausible to acknowledge that (i) there is a plurality for both, 
while (ii)  the divergence of options and opinions regarding semantics of numbers exceeds the divergence of views 
regarding notations for numbers. 
\end{claim}
The skewed plurality phenomenon indicates that uniformity and consensus is skewed 
in the direction of syntax, form, and notation, with high uniformity and consensus on such matters 
in comparison to consensus on meaning. I will provide some arguments for Claim~\ref{SPCss} below.
%

\subsection{Three levels of elementary arithmetical competence}
 I will assume that an individual $P$ upon being taught arithmetic
acquires what I will call  competence of elementary arithmetic (CoEA). 
Such competence is probably best measured as some
partial ordering. Instead I will distinguish three levels of CoEA. Being easily able to answer standard questions 
at elementary school level amounts to basic CoEA, 
whereas the competence level of a teacher is denoted as professional CoEA and the level of a 
professional mathematician is labeled as academic CoEA.

\section{Semantic options and issues for naturals and integers}
\label{SurveySem}
Ontology of arithmetic may be based on various preconceptions, to mention:
\begin{itemize}
\item Natural numbers can be considered ordinals as well as cardinals, and both views suggest different semantic options.
\item Natural numbers, if there are supposed to exist, are interpretations of suitable AQs and therefore the ontology of natural numbers may also be 
derived from syntax if syntax is supposed to exist in advance of arithmetical semantic considerations.

\item Naturals are integers. In mathematics the informal 
use of language suggests that the natural numbers are subset of the integers. 
Although the assertion ``$3$ is an integer?'' is plausibly confirmed, strictly speaking most
conventional definitions of naturals and integers do not support this assertion. 

\item It is plausible to view the collection of natural numbers as a set, with as a consequence upon adopting ZF set theory that particular natural numbers are sets as well. However, once taking ZF as the technical foundation, any subsequent choice for a set 
$N$ of naturals is quite arbitrary, and labeling a particular set as a natural number 
(on the grounds of membership of $N$) can only be based on an arbitrary convention.

\item One may hold that natural numbers can and should be defined without a simultaneous introduction of the core operations of addition and multiplication, the introduction of which takes place at a later stage. And one may 
alternatively hold the converse, claiming that without operations there is no point in viewing any objects as numbers, i.e. that natural numbers must come with an arithmetical datatype.
\item Contemplating the natural numbers as a collection is an option. 
Collection is a form of aggregate which is more general than set and class. 
The natural numbers may als be chosen as urelements for a ZF style set theory.
\end{itemize}

\subsection{Ordinal natural numbers (finite ordinals in ZF)}
A well-known construction due to J. V. von Neumann~(\cite{Neumann1922}) determines the natural numbers (i.e. the elements of the set $N$
which itself is an ordinal as well) as the finite elements of the class of ordinals.

Now following the 
observations of Paul Benacerraf in~\cite{Benacerraf1965}  numbers of sort $nat$ 
are not intrinsically unique abstract mathematical entities,
at least not in ZF style set theory and under the assumption that $N$ is a set, 
because, in spite of the technical efficiency of the well-known encoding as finite ordinals, 
there is no unique canonical representation of ordinal numbers in ZF.

The von Neumann encoding of ordinals provides finite ordinals as a very plausible structure for 
naturals but that design is not unique, and there is no compelling reason to adopt that particular design so that there is no compelling reason to view the finite ordinals thus defined as the true natural numbers. 
%
I will refer to naturals thus defined as the ordinal natural numbers, notation $N_{ord}$, with the understanding that ``ordinal'' will be omitted on most occasions. 
\subsection{Cardinal natural numbers (finite abstract cardinals outside ZF)}
For cardinal numbers one may distinguish concrete cardinal numbers, that is ordinals which are cardinals at the same time, 
and abstract cardinals, i.e. the class of all sets with the same number of elements. Thinking in terms of concrete cardinals leads to the same natural numbers as thinking in terms of ordinals. Considering abstract ordinals that changes.

Working in ZF set theory the sets with cardinality $17$ constitute a proper class
say $C_{\#17}$. Let $\omega$ denote the set of finite ordinals in ZF, and let $\equiv_\#$ denote equicardinality on sets. 
Then $C_{\#17} = \{x \mid x \equiv_\# \underline{17}\}$ with $\underline{17}$ the 
$17$th ordinal in $\omega$. 
I will write $[A,B,C,\ldots ]$ for a collection the elements of which are classes and may be proper classes.
Now $N_{card} =[C_{\#0}$, $C_{\#1}$, $C_{\#2}, ...]$ is a collection of 
classes which itself is not a class, because its elements are proper classes. 
I will refer to a collection of classes as a meta-class; meta-class is a special case of collection. 

The meta-class $N_{card}$ cannot be found in ZF set theory but it exists in variations thereof, such as
Ackermann's theory of sets and classes (see e.g.~\cite{Muller2001} for a discussion of these matters). 
One may write
 $$N_{card} = [X \in \mathsf{Class} \mid \exists n \in N_{ord} .(X = \{x \in  \mathsf{Set} | x \equiv_\# n\})]$$
 where $X$ ranges over all classes and $x$ ranges over all sets. 
 Forgetting that the definition of $N_{card}$ uses $N_{ord}$ one may take the extension of $N_{card}$ as the more fundamental notion which is available for conceptualising naturals.
 However obvious the idea of $N_{card}$ may be, when contemplating finite cardinals from first principles, 
 there is not even a standard notation for it as the use of $[..]$ is merely an ad hoc proposal for this occasion, 
and it lies outside the comfort zone of ZF based academic mathematics. Nevertheless  $N_{card} $ 
constitutes a credible option for an ontology of natural numbers.

\subsection{Peano's natural numbers}
A classical idea due to Peano is to have a successor function $S$ in mind and to identify the sequence of naturals as the sequence: $0,S(0), S(S(0)), S(S(S(0))),\ldots$. Now $17$ is a notation for a natural number rather than a number itself.

As a conceptual idea the main source of lack of uniqueness of number thus conceived comes from the choice of a name
for the successor function. This price must be paid if one prefers not to consider naturals as a collection which is 
given (and thereby unique) as a a primitive set in advance. The Peano representation is technically very attractive,
and has found many applications in theoretical computer science. This approach and 
can be used to provide the semantic foundation of elementary arithmetic without taking all of ZF on board. 
One finds e.g. $ \llbracket 3 \rrbracket = S(S(S(0)))$. One may view ``$3$'' as belonging to syntax and ``$S(S(S(0)))$'' 
as a semantic entity. I will refer to these natural numbers as Peano's natural numbers and I will refer to Peano's natural numbers as $N_{peano}$.

\subsection{Decimal natural numbers: projection semantics}
\label{EclNat}
An interpretation of $\Nat$ which is workable for basic CoEA as well as for professional CoEA 
 identifies the natural numbers simply as non-empty digit sequences. So it is assumed that
 $N_d \subseteq $ AQ, $N_d \subseteq $ AQ, and $N_d \subseteq $ AQ.
 A naive set theory is adopted which does not give rise to the idea that digit sequences should be encoded
in a setting of lower level primitives, and therefore does not come with a sense of arbitrariness of a specific encoding of these.
With this interpretation semantics of natural numbers becomes a projection of AQs.

The elements of $Z_d$ may be termed decimal natural numbers in order not to claim that a fully general treatment of 
natural numbers has been obtained and suggesting instead that no such fully general treatment exists.

Thus a single structure and domain for the interpretation of the sort $nat$ is chosen. 
Insisting on ZF foundations the pluralistic nature of this approach is undeniable because encoding digit sequences in ZF can be doe in many ways, none of which can claim a foundational significance in excess of other encodings.

I will adopt (this particular form of) projection semantics for naturals and integers as a preferred option, and I consider this choice to be justified by a combination of its intuitive appeal 
with the absence of any other known option for the semantics of natural 
(and integers) which is  conceptually more convincing. 

\begin{definition} 
\label{DecNats}(Decimal natural numbers). $N_d^+$ is the 
collection of non-empty sequences of decimal digits starting with a 
non-zero digit as the first symbol. $N_d = \{0\} \cup N^+_d$. $N^+_d$ is understood 
as the collection of positive decimal natural numbers, and $N_d$ is correspondingly understood as the collection of 
decimal naturals.
\end{definition}
 
Naive set theory allows not to be bothered by the details of encoding sequences of digits in ZF set theory. Such details unavoidably come with many alternatives thereby refuting  uniqueness of the concepts in a principled sense. 
The fact that at some stage a confrontation with classical paradoxes 
necessitates more caution is then taken for granted as these paradoxes are unlikely to show up in school arithmetic.

An argument in favour of definition~\ref{DecNats} is that it is as close as possible to the idea that ``$17$ is a natural number''.
Using naive set theory one may claim that all occurrences of the sign $17$ share the underlying non-material meaning of $17$-hood which is codified by Definition~\ref{DecNats}. From the point of view of logic or of computer science
it is a disadvantage of decimal natural numbers that there are infinitely many constants. From a conceptual perspective this disadvantage may considered an advantage instead because it is simpler. 
By taking only a finite collection of constants
into account one obtains useful simplified and finite approximations of arithmetic.

\subsection{Semantic options for integers}
Given the naturals, say $N_{ord}$, $N_{card}$, $N_{peano}$, and $N_{dec}$ different constructions for integers 
can be applied used. I will mention some options without any claim of completeness. 
Below $\alpha$ ranges over $\{ord,card,peano,d\}$.
\begin{enumerate}
	\item Consider the integers as equivalence classes of pairs of naturals (with $(a,b) \equiv (c,d) \iff a+d = b + c$). 
	This leads to $Z_{ord}^{eqc}$, $Z_{card}^{eqc}$, $Z_{peano}^{eqc}$, and $Z_{d}^{eqc}$. 
	As a consequence of these assumptions:  $N_\alpha \cap Z_\alpha = \emptyset$.
	\item As a second option one may understand 
	the integers as equivalence classes of pairs of naturals (with $(a,b) \equiv (c,d) \iff a+d = b + c$) and adapt the 
	naturals to this notation by 
	redefining the natural numbers as the equivalence classes of the form $[(n,0)]_\equiv$ for $n \in N_\alpha$, 
	that is the naturals inside the integers. Write $N_\alpha^{eqc}$ for this set (with $eqc$ for equivalence classes).
	And as a consequence of these modified assumptions: $N^{eqc}_\alpha\subseteq Z^{eqc}_\alpha$. 
	
	
	\item Adopt  and in addition consider the positive naturals 
	 $N^+_\alpha =N_\alpha -\{0\}$ (where the mechanism for deleting $0$ differs for the respective cases)
	adopt $Z^s_\alpha= \{0\} \cup \{0,1\} \times N^+_\alpha$ ($s$ for signed) as 
	definition of the integers 
	on top of the naturals. As a result $N^+_\alpha \subseteq N_\alpha $ and $N^s_\alpha \cap Z^s_\alpha = \{0\}$.

	\item (Decimal integers.) In the case of decimal integers each of the mentioned options may work, 
	but the following is more plausible, and the simplest notation is reserved for this case:
	 $Z_d = \{0\} \cup N_d^+ \cup-N_d^+$. 
	It then follows that as sets: $N^+_d \subseteq N_d \subseteq Z_d$. 
	\end{enumerate}

The plurality of options for the integers constitutes a cascade starting with a plurality of options for natural numbers. 
In no way, however, the survey of options is complete or can be completed.

I notice that decimal rationals may be defined as the decimal integers $Z_d$ extended with the AQs $p/q$ and $- p/q$ where: (i) $p$ and $q$ are nonzero decimal naturals, (ii) $q > 1$, and (iii) $p$ and $q$ have no common divisors beyond $1$; 
with this definition  decimal rationals constitute a subcollecton of AQs while being quite different from so-called decimal fractions. For rationals this leads to the set of numbers $Q_{d}$ for which $N_d \subseteq Z_d \subseteq Q_d$ is valid.
\section{Arithmetical datatypes}
The collections of numbers outlined in the preceding Paragraphs each may be equipped with one or more 
operations and constants taken form: successor ($S(_)$) addition 
 ($\_+\_$), opposite ($\_-$), multiplication ($\_\cdot \_$) and constants $0$ and $1$. 
 Combining a domain with a collection of constants and functions gives rise to the idea of a datatype, 
 and in the present context I wil speak of arithmetical datatypes, thereby loosely indicating properties expected of these constants and functions. For instance (i) $N_d(0,1,+,\cdot)$ is the datatype of decimal naturals 
 (which already contain $0$ and $1$) with addition and multiplication, (ii) $Z^s_{ord}(0,1,+,-,\cdot)$ is the datatype of signed integers based on ordinal naturals, (iii) $N_{peano}(0,1,+)$ is the arithmetical datatype with Peano's naturals and a constant $1$, equipped with addition. (iv) $Z_{card}^{eqc}(0,1,+,-)$ represents integers obtained as equivalence classes over cardinal naturals, with constants $0$ and $1$ and operations addition and opposite. These functions are not sets, and are merely defined as collections of pairs.
 
 \subsection{Isomorphism classes of arithmetical datatypes}
 The structures $N_{ord}(0,1,+),N_{peano}(0,1,+)$, and $N_{d}(,1,+)$ have much in common, in fact:
 $N_{ord}(0,1,+) \cong N_{peano}(0,1,+) \cong N_{d}(0,1,+)$ i.e. the three structures are isomorphic. $N_{card}(0,1,+)$
 is in the same isomorphism class (or rather isomorphism collection)  but a ZF based definition of isomorphism with the other structures cannot be used.
 
 Similarly $Z_{ord}^{eqc}(0,1,+,-) \cong Z_{peano}^{eqc}(0,1,+,-) \cong Z_{d}^{eqc}(0,1,+,-) \cong$\\
 $ Z_{ord}^{s}(0,1,+,-) \cong Z_{peano}^s(0,1,+,-) \cong Z_{d}^s(0,1,+,-) \cong Z_{d}(0,1,+,-)$. An isomorphism class of a 
  datatype is an abstract datatype. I will use the following notation:
  $\Nat(0,1,+)$ is the isomorphism class of $N_{d}(0,1,+)$, $\Int(0,1,+,-)$ is the isomorphism class of 
  $Z_{d}^{s}(0,1,+,-)$, $\Int(0,1,+,-,\cdot)$ is the isomorphism class of 
  $Z_{d}^{s}(0,1,+,-\,cdot)$. Working with arithmetical abstract datatypes removes much of the pluriformity which has 
  been noticed above, leaving the signature, i.e. the listing of constants and operations as the major remaining parameter.

Structuralism (in elementary arithmetic) refers to the viewpoint that isomorphism classes rather than particular structures (arithmetical abstract datatypes rather than arithmetical datatypes) constitute the fundamental ontology of number systems.
\subsection{Pluriformity for naturals and integers}

\begin{proposition} (Objective number scepticism). There is no such thing as a natural number. 
At best there are isomorphism classes of structures for 
natural numbers, assuming the presence of a first element named $0$, a second element named $1$, and a successor function (which may be derived from addition), a number being a mapping from structures in the isomorphism class producing an element of the domain of the argument (that is a selection of an element from each domain) in such a manner that for each structure equally many successor steps are needed to arrive at the selected element 
(this class of numbers allows an inductive definition). 
\end{proposition}
Number realism in each of its forms seems to involve choice between a plurality of options. 
Simply adopting number realism, say for $nat$, does not create much clarity in excess of the fact that having done so
a choice for a semantic domain $nat$ is in order.

\subsection{Structural realism}
In response to objective number scepticism one may insist that structures rather than numbers 
exist be it modulo isomorphism. However, mathematical structuralism (realism for mathematical structures, see e.g.~\cite{Muller2010}) based on ZF set theory is pluralistic just as much as number realism is bound to be. 
It is not clear to what extent structuralism can be decoupled from explicit mention of signatures, as the remaining 
source of pluriformity for number systems. Doing away with signatures leads to the following description.
\begin{proposition} 
\label{structuralRealism}
(Structural and pluralistic realism for $\Nat$). Assuming ZF set theory as a basis it is plausible to support pluralistic objective structural number realism for $\Nat$: structures for numbers of sort $\Nat$ exist and are unique up to isomorphism.
With that understanding of say $\Nat$ an assertion like $2+2 = 1 + 3$, expresses a fact about a chosen structure for 
$\Nat$, and indeed about all isomorphic structures for $\Nat$, an assertion which is acceptable for all agents who 
share the chosen structure (or structures).
\end{proposition}
Structural and pluralistic realism suggests to elevate arithmetical abstract datatypes to the level of first class semantic entities, rather than arithmetical datatypes proper.
%
\subsection{Abstract pointed arithmetical datatypes}
\label{PSC}
$\Nat(0,1, +)$ consists of all infinite commutative and associative semigroups in additive notation
 which are generated by $1$. Now the semantics ($\llbracket 17 \rrbracket$) of the number 
 denoted by $17$ (assuming that one wishes to view $17$ as denoting a numer, rather than as a number itself) may be taken to be the 
mapping (univalent relation) which assigns to a structure in $\Nat(0,1, +)$ the  $17$-th element of its domain. 
This kind of mapping can be formalised in terms of abstract pointed structures.
$\Nat(0,1, \widehat{17},+)$ as the isomorphism class of $N_d(0,1, \widehat{17},+)$ with $\widehat{17} $ interpreted as $17$. Here $N_d(0,1, \widehat{17},+)$ is a so-called pointed structure, in this case with a constant $\widehat{17} $ pointing to $17$.

Using abstract pointed arithmetical datatypes as the foundation on an ontology leads to the following ``definition'' of 
the natural number $17$: $\llbracket 17 \rrbracket_{apadt} = \Nat(0,1, \widehat{17},+)$.

\subsection{Beyond ZF ste theory  1}
While one might think of the naturals as a paradigmatic example of an infinite set, 
and a proof of concept for the very idea of infinite sets, it is not. It is the virtue of arithmetical datatypes to 
sets as domains. By changing one's set theory, for instance by adopting 
Ackermann's set theory as mentioned above collection of numbers may be turned into a class,
 though not into a set.

\subsection{Beyond ZF set theory 2: HTT}
One may be dissatisfied with the idea that pluralistic structural realism (or its meta-class version) is the final word on what a natural number is. Beyond such ideas modern mathematics has a lot to say, however. In the 90's of the 20th century Voudvovsky has initiated a novel approach to the foundations of mathematics which combines type theory with dependent types (thereby following de Bruijn and Martin-L\"{o}f), proof checking based on type theory (e.g. following Coq), constructive logic (Heyting) and mathematics
(Brouwer), category theory, 
and homotopy theory. According to~\cite{Awodey2014} so-called homotopy type theory (HTT) overcomes the arbitrariness of ZF 
set theory as a universal language for encoding mathematics, 
and opens the door to a more determinate, and therefore more credible, 
form of structural realism in mathematics than is provided on the exclusive and classical foundations of ZF set theory.

Moving beyond ZF style axiomatic set theory provides perspective on various forms of realism which 
seems to be blocked in set theory.
\begin{perspective} Homotopy type theory opens a path towards (univalent, non-pluralistic) 
structural realism for various number sorts, including $\Nat$ and $\Int$.
\end{perspective}

Once structural realism is feasible for numbers the fact that individual numbers can be uniquely determined 
also provides a path towards number realism.
\begin{perspective} Homotopy type theory opens a path towards (univalent, non-pluralistic) 
realism for various number sorts, including $\Nat$ and $\Int$ which will be types (rather than classes or sets).
\end{perspective}
HTT provides a potential path forward  for acquiring an academic CoEA for individuals who are not satisfied with the lack of determinacy of a pluralistic structuralist approach within ZF, who consider meta-class realism unattractive because meta-classes are uncommon entities, and who prefer not to become dependant on relatively unknown 
alternatives for ZF set theory.

%
%
\section{Arithmetical quantities}
Independently of one's view on semantics, one may or may not accept the existence of syntax. In the context of elementary arithmetic, I will identify the existence of syntax with the existence of AQs. I will 
use syntax realism as the label for a position which admits the existence of AQs. 
An attempt to describe in more detail what syntax realism amounts to lead to the following 
Definition, which works conditionally on whether or not semantic realism is adopted.
\begin{definition}
\label{SyntaxRe}
(Syntax realism.) $P$ adopts syntax realism for sort $S$, here chosen from $  \{pnat,nat,int \}$, if the 
following conditions are met: 
\begin{enumerate}
\item $P$ uses a name, say $AQ_S$, for a subtype of AQ.
\item  $P$ assigns to $AQ_S$ as its meaning (extension) a 
collection $\small \mathsf{AQ}_S $ of non-material entities which are referred to 
as AQs (terms, expressions) of (for) sort $S$.
\item If $P$ assumes some form of realism for sort $S$ (by interpreting $S$ as $||S||_P$)  then to each AQ $p$ in  
$\small \mathsf{AQ}_S$ an element $\llbracket p \rrbracket$ of $||S||_P$ is assigned.

\item The elements of $\small \mathsf{AQ}_S$ are referred to as terms for $S$ (alternatively: terms of sort $S$, 
expressions for $S$, or AQs for $S$).
\end{enumerate} 
\end{definition}

If the assignment function $\llbracket - \rrbracket$ is surjective 
$\small \mathsf{AQ}_S$ is considered to be complete 
(from the perspective of $P$) for the description of numbers of sort  $S$.

Adoption of syntax realism for $S$ may but need not imply that $P$ considers $\small \mathsf{AQ}_S$
to be a set on which (mathematical) functions can be defined. Neither is it always the case that an 
equality relation is present for 
$\small \mathsf{AQ}_S$. If present I will assume that this equality relation is denoted with (or can be referred to as)  
$\isaq$, or if confusion with other sorts may lead to confusion as $\isaq^S$. 
The various sorts may be viewed as a type system for AQs which allows polymorphism, that is AQs 
having different types at the same time. Polymorphism is the rule rather than the exception:
\begin{claim} 
For each agent $P$ who adopts syntax realism for each sort $S \in  \{pnat, nat, intt \}$ it is the case that:
 $\small \mathsf{AQ}_{pnat}  \subseteq \small \mathsf{AQ}_{nat}  \subseteq \small \mathsf{AQ}_{int}  $.
\end{claim}

AQs provide names for numbers, and words may provide names for some AQs, though such relations raise questions too. 
For instance: what is the relation between `zero' and $0$, and between `one' and $1$? 
Are one and $1$ synonyms, i.e. different names for the same number, or is one a name for $1$, 
just as `Einz' is a name for it, though in another language?

\begin{claim}
Assuming number realism for sort $S$ then syntax realism for sort $S$ is pluralistic.
\end{claim}
Indeed I  believe that whoever adopts syntax realism in addition to number realism for sort $S$, 
will also believe that pluralism for syntax cannot be avoided. 
There will always be different ways  to denote a given (semantic) entity, and also different ways 
to denote semantic objects in a simplest manner, given some conventions regarding simplicity of AQs,  given a plausible class of AQs.
Curiously by doing away with number realism syntax pluralism can be avoided at the same time. By making a distinction between AQs and their meaning, specific AQs become mere options for denoting that meaning. In the absence of number realism, however, $2+ 2 = 4$ is merely one of many assertions which is derivable in some calculus and is not an assertion about the respective meanings of $2+2$ and $4$.
\begin{claim}
\label{ClaimSOE}
Syntax realism is open ended.
\end{claim}
Unlike number realism where the entities serving as numbers are clearly demarcated in each approach, 
syntax realism need not come with sharp boundaries on what constitutes an expression. 
In particular syntax realism may come with a context providing definitions and abbreviations allowing the creation of 
additional expressions. Open endedness involves vagueness. For instance,  a context may introduce 
$f(\_)$ as a function of type $nat \to nat$ and then one may claim that $f(2) + f(3)+1$ is in $\mathsf{AQ}_{nat}$. Some may object that more information on $f$ is needed. If $f(\_)$ is introduced a partial function of type $nat  \to nat$ (for instance the predecessor function $P(\_)$ under the assumption that it is undefined on $0$), is it then valid to claim that $f(2) +1$ is an AQ.
Doing so or not doing so is determined by conventions on how to deal with undefined expressions and such conventions are not unique. One option for dealing with partiality  is to adopt  the convention that whether or not 
$f(2) + 1$ is an AQ  depends on the existence of $f(2)$ which must be settled first. If $f(2)$ does 
not exist $f(2) + 1$ is not an AQ, and if it exists $f(2) +1$ is an AQ. This idea makes syntax dependant on semantics, a significant complication.

\subsection{Sumterm splitting}
Although sumterms do not constitute a set it is intuitively covincing that the components of a sumterm can be 
selected by means of suitable selection operators. The collection of sumterms not being a set, the sumterm splitting operations cannot be considered functions, as the domain and the range of functions are required to be sets, the graph of a function itself being a set.
\begin{definition}
\label{SumtermDOs}
(Sumterm splitting operators.)
The pair $l_s$ and $r_s$ constitutes operations able to decompose a sumterm in its two parts 
so that for all AQs $X$ and $Y$,
for which $X+Y$ is a sumterm: $l_s(X+Y) \isaq X$ and $  r_s(X+Y) \isaq Y$. 
\end{definition}
For instance with $X \isaq 1+2$ and  $Y \isaq 3$ the condition that $X+Y$ is a sumterm fails
 as $1+2 + 3$ is not a sumterm (it is a poly-infix sumterm, however). I assume that if $X$ is not a
sumterm we have $l_s(X) = r_s(X) = 0$. The following implications are valid: $X \isaq Y \implies l_s(X) \isaq l_s(Y)$ and
$X \isaq Y \implies r_s(X) \isaq r_s(Y)$. One may claim that a sumterm has three parts, 
with the addition operator constituting a third part. The later claim, however, conveys no new information for an 
entity which is known to be a sumterm.

The presence of $\isaq$ next to $=$ may be considered overdone and useless, and  the distinction between both equality signs may simply be ignored, a line of thought which is taken in mathematics and education throughout. 
This simplification may be called the ``sums are terms'' paradigm.

\subsection{The sum splitting paradox}
I will now assume that the sums are sumterms paradigm has been adopted and that sum rather than sumterm is used. 
Instead of sumterm splitting operators there are sum splitting operators, though with the same definition.

A naive understanding of sums as (semantic) entities which can be split in parts 
 called summands leads to  a paradox. 
 This paradox will be referred to as the sum splitting paradox, which is an instance of phenomenon (AQ related decomposition paradox) which may 
 arise for other arithmetical operators just as well. 
\begin{proposition} 
\label{SDP}
Without making a distinction between AQs and values, or between $=$ and $\isaq$, the 
very presence of sum splitting operators leads to an inconsistency in arithmetic, 
in particular, it allows to infer $1 = 2$.
\end{proposition}
\begin{proof}
Assume the presence of $l_s$ and $l_r$ as in Definition~\ref{SumtermDOs}. Not distinguishing AQs and values
one may refer to these operations as sum splitting operations, and assume that the following implications are valid: 
$l_s(X+Y) = X$ and $  r_s(X+Y) = Y$. Now taking $X =1$ and $Y =1$ yields $1 = l_s(1+2) = l_s(2+1) = 2$. 
\end{proof}
\begin{claim}
As $1=2$ is an unacceptable conclusion, either some of the (implicit or explicit) assumptions for Proposition~\ref{SDP} or 
 some part of the argument in the proof of Proposition~\ref{SDP} must be dismantled in any sound approach to arithmetic.
\end{claim}
Ways of dismantling the argument for Proposition~\ref{SDP} will be referred to as 
solutions of the sum splitting paradox.

\section{Five solutions of the sum splitting paradox}
\label{FiveSol}
I will distinguish five solutions of the sum splitting paradox each of which remove it as an 
obstacle for the credibility of an account of elementary arithmetic of naturals and integers.

Each of these views can be combined with the idea that sum is a role for a value (see Paragraph~\ref{Roles} above),
rather than a particular kind of values or other entities. It is only in the matter of explaining the notion of a summand that
the viewpoints below differ.

\subsection{Sumterms: taking more detail into account}
The sumterm solution (for the sum splitting paradox) amounts to making a clear distinction between syntax and semantics for arithmetic.  Using sumterms the issues involved can be discussed in more detail
 with the effect that what seems to be a paradox at first sight is in fact a mere a misunderstanding, coming about from a mere confusion of different notions of equality.

Sumterms are among the arithmetical quantities which are not arithmetical values (decimal numbers) at the same time.
Now it is claimed that sumterms instead of sums which admit decomposition by way of $l_s$ and $r_s$. 
The splitting operations  $l_s$ and $r_s$ 
are viewed as operations on AQs. It follows that substitution of equals by 
equals may fail, in the above example  $l_s(1+2) = l_s(2+1)$ cannot be inferred from $1+2 = 2 + 1$. 
The corresponding  implication which can be maintained instead is $1+2 \isaq 2 + 1 \implies l_s(1+2) \isaq l_s(2+1)$.
But $1+2 \isaq 2 + 1$  is not valid, and in particular it  does not follow from $1+2 = 2 + 1$.
 The derivation of $l_s(1+2) \isaq l_s(2+1)$ which occurs in the proof of the sum splitting paradox fails and as a consequence the paradox disappears.

 In the framework of the sumterm solution, the introduction of the functions $l_s$ and $l_r$ is considered unproblematic 
 on the following grounds: (i)  sumterms are pairs, (ii) the components of a pair can be arbitrary 
 entities (including AQs, even  in the absence of a rigorous definition of AQs, which may 
 conceivably found or constructed outside conventional set theory, and (iii) that pairs can always be decomposed, 
 a principle which is prior to ZF set theory. 
 
 The sumterm solution comes with the necessity to distinguish between $=$ and $\isaq$, 
 though of course not with the commitment to the use of the particular notation 
 chosen for $\isaq$ in this paper.

\paragraph{Assessment.} 
The sumterm solution constitutes a plausible option for doing away with the sum splitting paradox. 
The sumterm solution comes with a high price, however, and that is to take the distinction between syntax and semantics seriously and to allow operations on expressions (AQs) which have no counterparts in the world of values (numbers).

\subsection{Contradiction tolerance based on a  foundational specification}
Contradiction tolerant solutions for the sum splitting paradox accept the presence of the argument for a contradiction,
while making sure that no extensive harm is done to the reliability of arithmetic. Paraconsistent logic may be adopted to gain condition tolerance in a very general style. However, I have been unable to find a convincing paraconsistent logic for dealing with the sum splitting paradox. Instead ad hoc strategies can be adopted which take the subject matter of elementary arithmetic into account. 
I will distinguish two different options for condition tolerant solutions of the sum splitting paradox. 

First it is assumed that arithmetic is cast in terms of the equational theory of an abstract arithmetical datatype, 
which is specified by means of an algebraic specification. This specification is given a
 foundational status, which means that its consequences, including negations of non-derivable closed identities, 
  may overrule 
any conclusion in the form of an equation (such as $1=2$) which has been derived by other means. If a proof is at odds with the implications of said foundational specification of the arithmetical datatype at hand,
 the proof is rejected and its conclusion is not adopted. 
A foundational specification which can be used in the case of $int$ is given in Paragraph~\ref{Fspec} below.

In order to avoid mistakes and as a method to prevent costly checks of results of proofs 
against the given foundational specification it is advised 
not to apply substitution of equals for equals at top level to arguments of $l_s$ and $r_s$. Meta-theory  
concerning the observation that this rule of thumb suffices to avoid the derivation of wrong conclusions is considered 
inessential and need  not be provided.

The foundational specification solution has important predecessors, for instance ZF set theory may be considered a foundational specification against the background of which naive set theory may be used in daily mathematical practice in such a manner that the risk of running into a form of the Russel paradox is not completely absent. 
A similar idea is put forward in Bergstra~\cite{Bergstra2020} in the case of fractions instead of sims. 
The foundational specification (solution) approach rejects 
the distinction of $=$ and $\isaq$. This approach is both practical and theoretically sound and does 
not involve any sophisticated consideration of syntax. 
It might well serve as the basis for teaching elementary arithmetic.

\paragraph{Assessment.} 
Contradiction tolerance seems to be an attractive feature of reasoning systems, whereas reasoning systems that 
have been designed in such a manner as to minimise the probability of deriving an invalid conclusion come with 
seemingly artificial restrictions which give rise to an unnatural look and feel.

The sumterm solution may be adopted as an informal device which helps to work in such a manner that conclusions drawn
in a contradiction tolerant framework based on a foundational specification will not be in contradiction with that specification.
This means that reasoning (in the condition tolerant setting)  must preferably 
be done in such a manner that it can be provided with additional detail 
(such as type distinctions between value and AQ) so as to lead to reasoning patterns that are sound for the sumterm solution.

\subsection{Contradiction tolerance with pragmatic justification}
Contradiction tolerance can be obtained without adopting a foundational specification by means of reliance on a body of practical experience. Then arithmetic, like any other area of human endeavour is viewed 
as being informal and fault prone, and only to be trusted to knowledgeable persons, able to avoid a range of 
well-known as well as lesser known mistakes.

The introduction of $l_s$ and $r_s$ may simply be rejected because apparently it leads to problems. 
But it may be accepted by those who know how to avoid such problems. There is nothing special about these matters which
suggests that avoiding mistakes requires anything else in excess of experience with the subject at hand.

   Other wrong inferences may include instances derived from making any of the following assumptions: 
  (i) $x \cdot x$ is always even, (ii) $x \cdot (-x) = 0$ (both cases 
  confusion addition and multiplication), (iii) $x \cdot (y + z) = (x \cdot y) + z$
  (misunderstanding of associativity),  
  and  (iv) $x+y$ is always larger than $x$ (failing to take negative values into account). 
  
Arithmetic may be considered as just any practical activity: avoiding pitfalls and making sound steps go hand in hand. 
For a car driver there are many mistakes which are better avoided, 
and thinking in terms of such mistakes as well as in terms of avoidance of these and similar mistakes is helpful. 
Arithmetic thus conceived consists of many rules of thumb on how to solve certain ``problems'' 
and of many guidelines on how to avoid mistakes of various kinds. The idea that following a logical path of reasoning 
is sure to lead to valid conclusion and that therefore purely logical reasoning can (and should) be applied is 
as distant as it is in playing chess. The idea that arithmetic is a manifestation  
of a fully reliable logic applied logic is rejected as lacking both contemporary evidence and historical justification.
 
\paragraph{Assessment} The idea that arithmetic is just like most areas of human competence, a matter of experience and 
training, and a matter of adopting successful patterns of behaviour while avoiding patterns likely to be less successful, is
attractive and cannot be rejected in a principled manner. 

The situation may be comparable with chess: one may teach a chess student a
classical course on chess with openings, 
tactics, and strategy, and with a sample of end games. The student ends up with a lot of rules of thumb in their 
mind, and some ability to apply these guidelines when playing a game. 
 At the same time a student may be explained that
 a suitably programmed computer may somehow quickly find out what is the best move while the student
 has no clue regarding how the algorithm, which is implemented in the computer at hand, works. 
 For the student playing chess is a matter of following 
 guidelines and  avoiding mistakes, even if it is known that a more principled approach exists, i.e. the one followed by an 
 advanced chess playing computer.
 
 In the case of arithmetic a successful fast track introductory 
 course may not even touch the observation that arithmetic can be done in a 
 waterproof manner based on rock solid logic without any notion of mistake and avoidance of mistakes.

A disadvantage of perceiving arithmetic, as a topic being taught, as a mere pragmatic collection of do's and don'ts is that it
may be unhelpful for the advancement of logical reasoning competence, in that respect 
a missed opportunity so to say.
\subsection{Conventionalism/traditionalism on function definitions}
One may hold that functions must exist in set theory and that for that reason any function must
 have a set of mathematical entities as its domain. 
The sums on which $l_s$ and $r_s$ are supposed to work do not 
constitute a plausible collection of mathematical entities in ZF set theory and
for that reason it is invalid to introduce these functions in the first place. Introduction of sum 
splitting operations is at odds with standard conventions and traditions 
and can be refuted for that reason.
As a consequence Proposition~\ref{SDP} disappears.

Conventionalism on function definitions refutes the very plausibility of the introduction 
of functions $l_s$ and $r_s$. Conventionalism on function definitions comes with the suspicion that sumterms do not
constitute a set which can be used in a mathematical argument. 

\paragraph{Assessment.} Conventionalism on function definitions may 
be adopted by anyone who considers the incorporation of the syntax of expressions in mathematics to be
a bridge too far, let alone the incorporation of logic, and to bring about 
an extension of elementary arithmetic beyond its natural limits. Conventionalism on function definitions leaves the 
common terminology of summands unexplained, and will argue that there is always a transition from the formalised mathematics to informal language the elements of which lack formalisation and that summand is one of the notions which
play an intermediate role and are not considered to be in need of mathematical explanation or definition.

\subsection{Conventionalism/traditionalism on arithmetic signatures}
Conventionalism on arithmetic signatures denies a participant of school arithmetic the right
 to introduce functions and in particular objects to the introduction of both sumterm splitting operators, 
 because of the absence of any convention or tradition supporting such introductions. The architecture (that is signature) of the language of elementary arithmetic is supposed to be determined in advance. The idea that elementary arithmetic inherits from mathematics or 
logic a systematic methodology for redesigning (or extending) itself is rejected.

Conventionalism on arithmetic signatures is less restrictive than conventionalism on function definitions as outlined above,
and instead of objecting to the introduction of sum splitting operations on methodological grounds this 
form of conventionalism  objects to the the introduction of sum splitting operators because these change the subject 
(architecture) of elementary arithmetic in a problematic manner.

The introduction of sum splitting operations $l_s$ and $r_s$ is considered foreign to the tradition in and is rejected
for that reason. The proper choice of operators is considered part of arithmetic, just as with any natural language, and no rules of engagement for the introduction of new functions are made explicit. 
In particular as (elementary) arithmetic is considered to be prior to mathematics as well as being prior to logic, no rules or conventions for the introduction of new functions, beyond the use of abbreviations for explicit definitions, are taken on board, and definitely not if these rule only appear in subsequent stages of the development of mathematics and logic. 

\paragraph{Assessment.} Conventionalism on arithmetic signatures is a perfectly valid way to deal with the sum splitting paradox. This solution has these disadvantages: (i)  it leaves the notion of a summand unexplained,
(ii) it provides no definition of the notion of a sum given the fact that all decimal numbers are sums.

\subsection{Foundational specifications of $\Nat$ and of $\Int$}
\label{Fspec}
The datatype DGS (digits) introduces a meta function for successor: with $d$ ranging over the constants 
$0, \ldots, 8$, $d^\prime$ stands for $1,\ldots,9$ respectively. the abstract datatype of DGS is the initial algebra of the 
specification in Table~\ref{DgS}. 

Table~\ref{Int} provides an initial algebra specification of the abstract arithmetical datatype $\Int(dec, +,-)$. All elements of $N_d$ serve as constants in this specification, which is indicated by writing $dec$ for these
constants. The specification contains infinitely many equations because $\sigma$ ranges over all of 
$N_d$. When omitting the opposite operator $-$, and both equations for it,  
an initial algebra specification of $\Nat(dec,+)$ results. The specification of Table~\ref{Int} is a simplified version of a specification in~\cite{Bergstra2020b}.

With some work the equations can be given an alternative orientation so that precisely the elements of $Z_d$ will appear as normal forms and a definition of the arithmetical datatype  $\Int(dec, +,-)$ is obtained. A foundational specification, however, may just as well be an abstract datatype specification.
 \begin{table}
\centering
\hrule
\begin{align*}
0^{\prime} &\equiv 1&3^{\prime} &\equiv 4&6^{\prime} &\equiv 7\\
1^{\prime} &\equiv 2&4^{\prime} &\equiv 5&7^{\prime} &\equiv 8\\
2^{\prime} &\equiv 3&5^{\prime} &\equiv 6&8^{\prime} &\equiv 9
\end{align*}
\hrule
\caption{$\mathsf{DGS}$: enumeration and successor notation of decimal digits}
\label{DgS}
\end{table}

\begin{table}
\centering
\hrule
\begin{align}
 \texttt{include}&\texttt{: Table DGS} \nonumber \\
	(x+y)+z 			&= x + (y + z)\\
	x+y     			&= y+x\\
	x+0     			&= x\\
	x + (-x) 			&= 0  \\
	-(-x) 				&= x\\
	d^\prime 			&= d+1 ~(\mathsf{for}~d \in \{1,2,3,4,5,6,7,8\})\\
	9	+1 			&= 10\\
	\sigma d + 1		&= \sigma d^\prime  ~(\mathsf{for}~d \in \{0,1,2,3,4,5,6,7,8\})\\
	\sigma + 1 = \tau \to \sigma 9 + 1	&=\tau 0
\end{align}
\hrule
\caption{Specification of $\Int(dec,+,-)$; $\sigma$ ranges over nonempty digit sequences (not starting with 0)} 
\label{Int}
\end{table}

\section{More solutions of the sum splitting paradox}
There is no way to force anyone to choose between the five solutions for the sum splitting paradox, or to design any other solution for a problem one may not recognise. Under the assumption that the sum splitting paradox requires a solution,
 it is easy to find alternatives and modifications of these five solutions to it, and combinations of these views are options as well. I will understand the five solutions as building blocks for the design of solutions, the virtue of which may be 
 a better proximity to certain intuitions.
 
 I wil first discuss a limitation on the sumterm solution. 
Although adopting the distinction between AQs and values with sumterms as a special case of 
AQs helps to avoid the problematic fact in Proposition~\ref{SDP}, it will not prevent that upon the introduction of a more detailed view of AQs similar issues arise. 
Repeating the same strategy will provoke an infinite regress rather than solve a problem. 
This observation is detailed in the following Paragraph.

\subsection{Stopping an infinite regress}
\label{StopInfReg}
Stopping an infinite regress may constitute an argument against adopting the sumterm solution. 
The argument runs thus: (i) suppose sumterms and the sumterm solution of the sum splitting 
paradox have been adopted. Now given an AQ, say $X$, let $\#_{bp}(X)$ be the number of bracket pairs in $X$. For instance $\#_{bp}(((1+2)+(2+0))+0 )= 3$.

As abstract entities, the expressions $0 $ and $ (0)$ are the same, so $0 \isaq (0)$, from which it follows that 
$\#_{bp}(0) \isaq \#_{bp}((0))$ which implies $1 = \#_{bp}((0)) = \#_{bp}(0)=0$. 
Apparently a paradox similar to the sum splitting paradox has arisen. I will refer to such paradoxes as AQ related paradoxes. 
This one may be called the bracket pair counting paradox.

Apparently adopting the sumterm solution and the introduction of AQs does eliminate the phenomenon of AQ related paradoxes, while these steps only help to avoid the sum splitting paradox. A plurality of solutions for 
the bracket pair counting paradox may be designed just as for the the sum splitting paradox. Assuming that a solution is chosen by introducing a more detailed view, then a further ramification of AQs may be introduced with an equality $\isaq^{bp}$ for bracket pair aware AQ equality so that (i) $0 \not \isaq^{bp} (0)$, and (ii) 
$X \isaq^{bp} Y \implies  \#_{bp}(X)  \isaq^{bp}  \#_{bp}(Y)$
instead of $X \isaq Y \implies  \#_{bp}(X)  \isaq  \#_{bp}(Y)$. With these preacautions the bracket pair paradox disappears.

At this stage, however, there is further room for AQ related paradoxes. For instance  counting the number of spaces in an expression with a new operator, say $\#_{sp}(X)$ will produce similar complications. 
For instance $\#_{sp}(1~+~2)= 2$ and $\#_{sp}(1\!+\!2)= 0$, while $1~+~2  \isaq^{bp}1\!+\!2$. It turns out that AQ related paradoxes will appear whenever a notation for AQs is used which is less abstract (i.e. provides more detail) 
than the AQs one has in mind. Continuing along the lines of ever more detailed 
views of AQs ultimately leads to an infinite 
regress, which is both unattractive and unconvincing.

Both contradiction tolerance (with or without foundational specification) and conventionalism (either for function definitions or for arithmetical signatures) provide a workable solutions for the bracket pair counting paradox.
My preference is to adopt sumterms and in addition to adopt conventionalism/traditionalism on arithmetic 
signatures as a justification for rejecting the introduction of a bracket pair counting operation.

\paragraph{Assessment.} Avoiding an infinite regress is a justified objective. It follows from this idea that the sumterm 
solution to the sum splitting paradox cannot be justified as a means to do away with AQ related paradoxes
in any general manner. The sumterm solution can be motivated only by its use to explain in a systematic manner 
a notion of a summand in relation to a notion of sum, while the sumterm solution can not be motivated
by the objective to get AQ related paradoxes out of the way for once and for all.

I hold that commitment to the virtue of avoiding an infinite regress plausibly comes with an acceptance of either contradiction tolerance or some form of conventionalism/traditionalism.

Avoiding an infinite regress can be combined with the adoption of sum splitting operations and adopting 
contradiction tolerance on the basis of a foundational specification at the abstraction level of $\isaq$. The details of such an approach are non-trivial because substitution has to be defined in a bracket pair counting compliant manner, which necessitates a redevelopment of equational logic from first principles. I will not develop such details below. Instead I will 
adopt conventionalism/traditionalism on arithmetic signatures as a justification for dismissing the introduction of 
the bracket pair counting operator $\#_{bp}(-)$.

Conventionalism on function definitions and/or on arithmetic signatures will be increasingly convincing 
when AQ related paradoxes are met at lower levels of abstraction. 
Functions which are introduced for the sole reason of exposing a potential contradiction may 
be rejected as foreign to arithmetic. Indeed, the fact that summands are somehow a well-known notion in arithmetic 
essentially contributes to the justification of the introduction of  sumterm splitting operators.

\subsection{Combinations and modifications of the five solutions}
\label{Combi}
Some combinations and minor adaptations of the above views are attractive. I will survey some options for such combinations.
\subsubsection{Delayed contradiction tolerance}
The sumterm solution may be adopted merely as a guideline on how to work on the basis of the
pragmatic contradiction tolerant solution. Using the sumterm solution in this manner amounts to requiring that 
each proof which is written in the context of a condition tolerant 
approach allows being refined by replacing zero or more  equality signs (=) by  in to $\isaq$ signs 
with the effect that a valid proof in the context of the sumterm solution is found. Adopting sumterms,
sum splitting operators, and AQ equality, as a means to avoid invalid reasoning patterns in a contradiction tolerant setting, can be considered a heuristic use of the sumterm solution. Alternatively this approach may be labeled delayed contradiction tolerance.
Adopting the sumterm solution as a heuristic device provides significant
guarantees that no equations or negated equations are derived which either contradict the foundational 
specification or which are in contradiction with outcomes obtained in standard practice.

\subsubsection{Delayed conventionalism/traditionalism}
Who opposes contradiction tolerance as a method for providing foundations of elementary arithmetic may prefer to combine the sumterm solution with either one of the forms of conventionalism, thereby accepting 
both sumterm splitting operators as being well-defined while rejecting further function introductions such as the 
counting of bracket pairs which was mentioned above. I will refer to this idea as delayed conventionalism/traditionalism.

\subsubsection{Viewing the foundational specification as potentially faulty}
One may portray the use of a foundational specification  as being 
potentially fault prone because axioms may be wrong and so may be proofs based on the axioms. 
This viewpoint rejects the distinction between both forms of contradiction tolerance and incorporates the first 
form (based on a foundational specification) into the second form (doing without a foundational specification), by viewing the foundational specification as merely one of many ways in which professional experience with arithmetic can be documented.

\subsubsection{The sumterm solution as a device in the background} 
One may reject the introduction of sum splitting operations 
(perhaps motivated as an instance of either or both forms of conventionalism) 
 and at the same time admit that sum splitting operations may be dealt with in the background, that is invisible for students, and then claim that a choice for a solution of the resulting problem (i.e. choosing one of the five options or suggesting yet another option)  is made in the background  as well, in such a manner that premature 
 disclosure to students can be avoided and that it does not matter much which choices have been made.
 
 \subsection{Time dependent choices and combinations of views}
 The views outlined above share the property that the same closed identities in arithmetic can be derived and therefore
 it may be considered unproblematic if one's views on the matter change in time. There is no need for anyone to choose between the options listed in Paragraph~\ref{Combi}. An option is to have different views, chosen from this listing, at different times, or to claim that each of these is acceptable and that one leaves the choice to the agents with whom one is communicating about arithmetic.

\section{Concluding remarks}
Sumterms as well as sumterm splitting operators have been introduced and it is shown that a naive approach to
these matters gives rise to paradoxical results which are referred to as the sum splitting paradox. 
The sum splitting paradox is considered an instance of AQ (arithmetical quantity, i.e. expression) 
 related paradoxes which arise in cases where (i) semantics is more abstract than syntax
in a context where (ii) certain syntactic differences are brought to the surface, in spite of being 
abstracted from in the preferred semantic view.

A survey is given of policies for avoiding the sum splitting paradox from occurring or for making sure that
the consequences of its occurrence are acceptable.

This work may be considered as belonging to the foundations of school arithmetic, where these foundations are 
primarily derived from, or biased towards, the  theories of datatypes and abstract datatypes in theoretical computer science.
Naturals and integers constitute instances of abstract arithmetical datatypes. 

Several conceptual difficulties arise when working out the details of the perspective just outlined. 
First of all syntax cannot be simply defined in advance, and we adopt 
the idea that AQs (arithmetical quantities) are a vague (fuzzy?) initial approximation of syntax, the elements of which 
come about most clearly on the basis of a given signature. In other words a distinction between syntax and semantics is assumed without either of these being initially defined in any detail. Assuming that AQs for natural numbers are given at some stage (such as $(91+0) +(12+3)$) one may ask for the meaning of these AQs. 
Remarkably mathematics, philosophy, and logic have neither 
produced a stable answer to that question nor shown that no such answer exists. 
The philosophical ontology of numbers seems to be an unfinished subject.

Apparently the educational literature on school mathematics contains very little explicit information 
about the nature (meaning, ontology) of natural numbers. In~\cite{Vinner1975} the naive Platonic approach (n.P.a.)  is advocated which assumes various types as given without making any attempt to giving definitions. 
In this work no distinction is made between defining a system of natural numbers up to isomorphism (i.e. an abstract arithmetical datatype for naturals), and defining
an arithmetical datatype for naturals.  In~\cite{Vinner1991} the same author provides a theory of definitions in a technical context, claiming that arithmetic and mathematics constitute a technical context. 
This suggestion is then made that definitions (of technical notions, say in arithmetic) are useful for the stronger (more interested) students but should not be forced upon students. In the case of natural numbers this idea suggests that the n.P.a. may give way for the so-called definitional approach for students who can appreciate such explanations.

After a survey of various options I propose 
(a certain choice of) projection semantics for AQs for naturals and integers, for instance arriving at the digit sequence $106$ as the semantics of the AQ $99+7$ and also as the semantics of itself. $106$ is called a decimal natural number in order to emphasise the 
fact that the details of decimals have not been abstracted away. In projection semantics for naturals,
decimal natural numbers are not decimal representations of natural numbers, but on the contrary no commitment is made to the existence of entities that qualify as natural numbers, while the existence of decimal natural numbers is assumed.

Secondly the notion of a sum is chosen to be that of a role name.
Summand, however, is not a role name, and neither is sumterm.
Remarkably roles play nearly no role in datatype theory, which seems to be a missed opportunity.

\subsection{Options for further work}
Many topics for further work can be formulated. In particular  the following questions which merit further efforts:

(i) To develop a theory of signs for naturals (natsigns) and integers (intsigns). Signs are not abstract entities and the philosophy of objects,
including mereology, will enter the discussion. Recent progress in mereology which may prove indispensable for 
the mereology of natsigns and intsigns are stage theory and worm theory,  see \cite{Sider1996}.

These approaches introduce aspects which seem to be entirely foreign to theoretical computer science 
because computer science has less focus on computers as material entities than on the abstractions which come with analysis and engineering.
(ii) To develop a theory of proterms (product terms) where factors are components of proterms that can be 
selected by means of suitable splitting operations. The situation is similar to the development of sumterms but there is an important difference: product need not be understood as a role name in the arithmetic of naturals and of integers as the products constitute the complement of the primes and because viewing prime as role name is not plausible, so is viewing non-prime as a role name.

(iii) To develop an informal logic of elementary arithmetic which allows incorporation of  
AQ related paradoxes and solutions of these.

(iv) To analyse the notion of a square. For integers squares are a credible class of numbers,
so that there is little incentive to view ``square'' as role, instead of as a name for a class of entities. 
Considering the real numbers, however,
 the situation changes: being a square is the same as being positive, and the term square more 
plausibly serves a role name. 

(v) In~\cite{BergstraP2020} so-called true fractions are investigated, which are certain elements of arithmetical datatypes with unconventionally few identifications between different fracterms. 
It remains to be seen whether or not a corresponding notion of a true sum can be defined.

%

\subsection{Connections with computer science and other work}
In computer science, in particular in connection with programming and software engineering the presence of syntax is so
prominent that it is hard to imagine that in mathematical syntax is much less important. 
In practical computer science it is well-known that adding mathematics semantics to programming is 
difficult because doing so hardly corresponds to the intuitions of the average programmer. 
In elementary arithmetic the situation is the other way around: adding an explicit coverage 
of syntax to semantics oriented thinking goes against the intuitions of the average professional, 
and is quite difficult, if only for that reason. 
In this paper I have tried to find some middle ground where both syntactic and semantic considerations play a role.

The terminology of datatypes and abstract datatypes fits well with my objectives. 
For instance the decimal natural numbers may be understood as a datatype 
which implements the abstract datatype of natural numbers. 
This terminology accommodates the observations that: 
(i) decimal natural numbers can be understood as well-defined abstract entities in a naive set theory, 
(ii) abstract natural numbers do not exist otherwise than as mappings from the datatypes in the 
abstract datatype to elements thereof, which is the mechanism of pointed structure classes of Paragraph~\ref{PSC}.
For (abstract) datatypes and algebraic specifications thereof I 
refer to~\cite{BergstraT1995,Goguen1989,EhrichWL1997}.

Having thus established that the notions of an arithmetical datatype and an abstract arithmetical datatype are of 
relevance for elementary arithmetic,  it is fitting to notice that conceptualising 
arithmetic in terms of datatypes and abstract datatypes  also leads to novel number systems: common meadows~\cite{BergstraP2014a}, 
transrationals as in~\cite{AndersonVA2007} wheels 
in~\cite{Setzer1997} and~\cite{Carlstroem2004}, and further options surveyed in~\cite{Bergstra2019b}.

Contradiction tolerance may be obtained via the use of paraconsistent logics. In~\cite{BergstraB2015} 
an application of paraconsistency in the setting of arithmetical datatypes is developed. 
An application of condition tolerance via paraconsistency is developed in detail in~\cite{BergstraM2017}.
For a survey on paraconsistent logics I mention~\cite{Middelburg2011} and ~\cite{Middelburg2020}. 
Nevertheless, in spite of the presence of a plurality of approaches to 
paraconsistent reasoning  I have not been able to develop a convincing solution of the sum splitting paradox 
along the lines of condition-tolerance as provided by paraconsistency.

Viewing AQs as names one may take names for a core notion in the development of elementary arithmetic. 
That path is taken in~\cite{Rollnik2009}. As it turns out, importing a theory of names into 
elementary arithmetic is a non-trivial overhead. Distinguishing syntax and semantics for arithmetic, may be done and undone repeatedly by the same person. Such steps have been investigated in~\cite{NicaudBG2001} where these
steps  are termed disassociation and association respectively.

In~\cite{Bergstra2020} fracterms are introduced as a syntactic counterpart of fractions. 
The work on sumterms has in part been done in the same style as the work on fracterms, though there are significant differences: in the current paper there is a fairly clear description of what (a) sum is while 
for fractions no corresponding view is developed in~\cite{Bergstra2020}. Unlike for fractions, about which a formidable literature exists, there seems to be no work in the educational literature where a definition of a sum of integers is provided in advance of the development of educational methods about it. Literature on sums of fractions abounds, e.g.~\cite{Howard1991}.
In the context of integer arithmetic, the focus in the literature is always on how to compute a sum, and not on what a sum is more broadly speaking.
\paragraph{Acknowledgement.} The work on this paper has profited from discussions with James Anderson (Reading), 
Inge Bethke (Amsterdam), Kees Middelburg (Amsterdam), Alban Ponse (Amsterdam), Stefan Rollnik (Rostock),  John V. Tucker (Swansea), and Dawid Walentek (Utrecht).

\end{document}